\documentclass[11pt,a4paper]{amsart}
\usepackage[utf8]{inputenc}

\usepackage{amssymb}
\usepackage{amsmath}
\usepackage{MnSymbol}
\usepackage{amscd}
\usepackage[all,cmtip]{xy}
\usepackage{setspace}

\usepackage[top=2.2cm, bottom=2.5cm, left=2.5cm, right=2.5cm]{geometry}

\newtheorem{thm}{Theorem}
\newtheorem{lem}[thm]{Lemma}

\newtheorem{prop}[thm]{Proposition}

\newtheorem*{claim}{Claim}

\theoremstyle{definition}
\newtheorem{defn}[thm]{Definition}
\newtheorem{rmk}[thm]{Remark}

\newcommand{\ZZ}{\mathbb Z}

\newcommand{\QQ}{\mathbb Q}
\newcommand{\HH}{\mathbb H}
\newcommand{\Qbar}{\overline{\mathbb Q}}
\newcommand{\CC}{\mathbb{C}}
\newcommand{\GG}{\mathbb G_{\mathrm m}}
\newcommand{\PP}{\mathbb P}
\newcommand{\RR}{\mathbb R}

\newcommand{\angr}{\langle 2r \rangle}

\newcommand{\Real}{\rm Re }
\newcommand{\Imag}{\rm {Im}}

\title[Rational points on Grassmannians]{Rational points on Grassmannians \\and unlikely intersections in tori}

\author[L. Capuano]{L. Capuano}
\address{
  L. Capuano and U. Zannier,
   Scuola Normale Superiore \\
  Piazza dei Cavalieri 7, 56126 Pisa (PI) \\
   Italy}
   \email{laura.capuano@sns.it\\
   u.zannier@sns.it}
\author[D. Masser]{D. Masser}
\address{D. Masser, Mathematische Institut, Universit\"at Basel \\
  Rheinsprung 21, 4051 Basel\\ 
  Switzerland}
\email{david.masser@unibas.ch}

\author[J. Pila]{J. Pila}
\address{J. Pila, Mathematical Institute, University of Oxford \\
  Oxford OX2 6GG\\ 
     UK}
\email{jonathan.pila@maths.ox.ac.uk}

\author[U. Zannier]{U. Zannier}
     

\onehalfspace

\begin{document}

\maketitle

\begin{abstract}
In this paper, we present an alternative proof of a finiteness theorem due to Bombieri, Masser and Zannier concerning intersections of a curve in $\GG^n$ with algebraic subgroups of dimension $n-2$. Actually, the present conclusion will give more uniform bounds which respect to the former statement. 
The proof uses a method introduced for the first time by Pila and Zannier
to give an alternative proof of Manin-Mumford conjecture and a theorem to count points that satisfy a certain number of linear conditions with rational coefficients. This method has been largely used in many different problems in the context of ``unlikely intersections''. 
\end{abstract}

\section{Introduction}

This paper concerns ``unlikely intersections'' in the context of algebraic tori. More speci\-fically, here
we treat the special but significant case of a curve defined over $\Qbar$ in a multiplicative group $\mathbb{G}_m^n$. In this case, 
Bombieri, the second and the fourth author proved, in \cite{Bombieri_Masser_Zannier_1999}, the following theorem:

\begin{thm}[{\cite[Theorem 2]{Bombieri_Masser_Zannier_1999}}] \label{BMZ}
Let $X$ be an irreducible curve in $\GG^n$ defined
over $\Qbar$ such that no non trivial monomial of the form $x_1^{m_1}\cdots x_n^{m_n}$ is identically constant on $X$. Then, 
there are at most finitely many points $P=(\xi_1, \ldots, \xi_n)$ 
on $X$ for which there exist linearly
independent vectors $(a_1, \ldots, a_n)$ and $(b_1, \ldots, b_n)$ in $\ZZ^n$ with
\begin{equation} \label{2-condition}
\xi_1^{a_1} \cdots \, \xi_n^{a_n}=\xi_1^{b_1} \cdots \, \xi_n^{b_n}=1. 
\end{equation}
\end{thm}

We shall describe an alternative proof of this theorem obtained through a method introduced in
\cite{Pila_Zannier_2008} to give a 
different proof of Manin-Mumford conjecture. The present approach, moreover, will give not only a new proof, but a strengthening of the result, concerning the uniformity of the estimates for the cardinality of the relevant finite set. To do that, we shall apply a new result, coming from ideas of the third author and 
described in full generality in the Appendix, about counting rational points on certain transcendental varieties of Grassmannians.

The original question of looking at rational points on transcendental objects was raised especially by Sarnak in an unpublished paper \cite{Sarnak} 
related to \cite{Sarnak_Adams},
where, in parti\-cular, he introduced some principles of this method in the special case of tori. 

Motivated by this question, several theorems on this topic were proved, by
Bombieri and the third author \cite{BP} in the case of curves, by the third author \cite{Pila04}, \cite{Pila05} in the case of surfaces 
and finally by the third author and Wilkie \cite{Pila_Wilkie} in the more general context of ``definable sets in o-minimal structures''. We will recall these results in Section \ref{concept_block}. \\

We have thought to be not entirely free of interest to present the mentioned alternative proof: on the one hand, this shows another application of the method in question;
on the other hand, the present proof avoids some delicate results of Diophantine approximation used in the original argument, applying instead some 
quite weaker versions. Such simplifications may result to be useful in other contexts. 

\begin{rmk}
The conclusion of Theorem \ref{BMZ} remains true if we only ask that no monomial as above is identically equal to $1$,
as conjectured in \cite{Bombieri_Masser_Zannier_1999} and first proved by Maurin in \cite{Maurin_2008} over $\Qbar$ and then reproved in \cite{BHMZ} 
and generalized also to the complex case in \cite{Bombieri_Masser_Zannier_2008}. 
We do not discuss here this more precise version. For a survey on these topics and, in general, on ``unlikely intersections''
also in different contexts, see \cite{Zannier12}.

We point also out that the present approach should lead to the more general version if combined with other Diophantine ingredients. Indeed, very recently Habegger and the third author proved,
in \cite{HP14},
the Zilber-Pink conjecture in the case of a curve in an abelian variety  using arguments in part related to the present ones.

Some uniform bounds for certain particular classes of curves in $\GG^n$ were obtained by P.\,Cohen and the fourth author in \cite{Cohen_Zannier_few}.
\end{rmk}

\section{Preliminaries}

An algebraic subgroup $H$ of $\GG^n$ of dimension $r\less n$ can be defined by
equations of the form $x_1^{c_1} \cdots x_n^{c_n}= 1$, where the vectors $(c_1, \ldots, c_n)$ run through a lattice $\Lambda_H$ of rank $n-r$ 
over $\Qbar$. For more, see for example \cite[pp.\,82–88]{Bombieri_Gubler}. 

Let us fix an irreducible curve $X$ in $\GG^n$, defined over $\Qbar$ and not contained in any translate of a proper algebraic subgroup $H$ of $\GG^n$
(which amounts to the assumption of Theorem \ref{BMZ}).
For every $r\ge 0$, we use the following simple notation, setting
\[ X_{(r)}= \bigcup_{\dim H\le r} (X \cap H) \quad \mbox{ with $H$ algebraic subgroup of $\GG^n$.}  \] 
In what follows, we use the height defined over $\GG^n(\Qbar)$ as 
\[ h(x_1, \ldots, x_n)= \max_{i=1}^n \{h(x_i)\}, \]
where $h(x)$ is the usual Weil’s absolute logarithmic height.
We shall use the following result:
\begin{thm}[{\cite[Theorem 1]{Bombieri_Masser_Zannier_1999}}] \label{bounded_height}
Under the previous assumptions,
the (algebraic) points in the set $X_{(n-1)}$ have bounded Weil height.
\footnote{The conclusion is false if we only ask that that no monomial $x_1^{m_1}\cdots x_n^{m_n}$ is identically equal to $1$, as already observed 
in \cite{Bombieri_Masser_Zannier_1999}.}
\end{thm}

The original proof of Theorem \ref{BMZ} in \cite{Bombieri_Masser_Zannier_1999} combines this bound on the height with 
a delicate extension of Dobrowolski's theorem by Amoroso and David 
\cite{Amoroso_David}. In the present approach, even the following weaker result of Blanksby and Montgomery \cite{Montgomery}, previous to 
Dobrowolski, shall suffice: if $\alpha$ is an algebraic number with $[\QQ(\alpha):\QQ]=d$, not a root of unity, then % exists an absolute constant $\gamma$ such that 
\begin{equation} \label{BM}
h(\alpha) \ge \frac{1}{52 d^2 \log 6d}.
\end{equation}
Actually, even a milder estimate of the form $h(\alpha) \ge C d^{-m}$
for some positive constants $m$ and $C$ would be sufficient for our purpose.
In the article, we will Vinogradov’s symbols $\ll$ and $\gg$
to denote inequalities up to an unspecified constant factor. The dependence of the constants will be specified if necessary.

\section{On the concept of ``blocks''} \label{concept_block}

In this section we shall briefly recall some results concerning the counting estimates by the third author and Wilkie and the notion of blocks introduced by the third author in \cite{Pila09} (see also \cite{Pila11}) and heavily used in the new proof of the main theorem. 

These results concern counting rational points of bounded height on suitable real varieties. In our application, such varieties will be compact subanalytic sets in some $\RR^n$. This category of sets fits into the more general notion of ``definable sets in o-minimal structures''. We will not present a formal definition of this; for references, see \cite{Pila_Wilkie}, \cite{Pila09} or \cite{van_den_Dries}.\\

If $a/b \in \QQ$, with $a,b \in \ZZ$, $b \gtr 0$ and $gcd(a, b) = 1$, we put $H (a/b) = \max(|a|, b)$ and consider the usual extension to $\QQ^n$ by setting $H(a_1 , \ldots , a_n) = \max_j \{H (a_j )\}$. 

For a set $X\subseteq \RR^n$, let $X(\QQ)$ denote the subset of points of $X$ with rational coordinates. Moreover, we define $X^{alg}$ as the union of all the connected semi-algebraic sets of positive dimension contained in $X$ and call it the algebraic part of $X$. The following holds:
\begin{thm}[{\cite[Theorem 1.8]{Pila_Wilkie}}] \label{Pila_Wilkie}
Let $X \subset \RR^n$ be a definable set in an o-minimal structure and let be $\epsilon \gtr 0$. 
There exists a constant $c(X,\epsilon)$ such that \[ |\{P\in (X\setminus X^{alg})(\QQ)\ |\ H(P)\le T \}| \le c(X,\epsilon) T^{\epsilon}. \]
\end{thm}
However, in some applications the algebraic part of the definable set in question is too big, so the previous theorem may be useless even if the rational points of bounded height are rather sparse. By the way, as already pointed out in \cite{Pila_Wilkie}, the proof of the previous theorem gives something more, namely that the rational points of height bounded by $T$ are contained in $\ll T^{\epsilon}$ ``connected pieces'' of the algebraic part. 
These ``pieces'' of the algebraic part, called ``blocks'', have a precise geometric form that we are going to define later.\\ 

Let us firstly focus on the following example. Take $\Gamma\subset \RR^2$ the graph of a transcendental curve $y=f(x)$ on a compact set and let $X$ be the product in $\RR^3$ of $\Gamma$ and the closed interval $[0,1]$, i.e.
$X=\{(x,y,z)\in \RR^3 \, |\, (x,y)\in \Gamma \mbox{ and } 0\le z\le 1\}= \Gamma \times [0,1]. $
If $f$ is analytic, $X$ is a subanalytic set. Furthermore, the algebraic part of $X$ coincides with the entire set: in fact, it is consists the union of the vertical lines $(x_0,y_0)\times [0,1]$ with $(x_0,y_0)\in \Gamma$.
So, the conclusion of Theorem \ref{Pila_Wilkie} is trivial as $X \setminus X^{alg}$ is empty. However, the rational points of $X$ of height bounded by $T$ are rather sparse. In fact, as $\Gamma$ is the graph of a transcendental function, we can apply Theorem \ref{Pila_Wilkie} (or the previous result \cite{BP}) to conclude that the number of rational points of $\Gamma$ of height bounded by $T$ is $\le c T^{\epsilon}$. Furthermore, if $\pi:\RR^3 \rightarrow \RR^2$ is the projection to the first two coordinates, the points we are interested in are contained in the preimages of rational points of $\Gamma$ of height bounded by $T$, i.e. in $\le c T^{\epsilon}$ vertical lines. These vertical lines are exactly the ``blocks'' which contain all the rational points of height bounded by $T$. \\
%The union of these ``blocks'' is a ``definable piece'' of the algebraic part of $X$, as it is a finite union of lines. \\

This example shows that the formulation of Theorem \ref{Pila_Wilkie} is, in some sense, not optimal. For this reason, the third author in \cite{Pila09} and \cite{Pila11} gave a refinement of Theorem \ref{Pila_Wilkie} introducing the following notion of ``blocks''. 
\begin{defn} \label{basic_block}
A definable block of dimension $k$ in $\RR^n$ is a connected definable set $U\subseteq \RR^n$ of dimension $k$ such that it is contained in a semialgebraic set $A$ of dimension $k$ and such that every point of $U$ is regular of dimension $k$ both in $U$ and in $A$. A block of dimension $0$ is a point. A definable block family of dimension $k$ is a definable family whose non empty fibers are all definable blocks of dimension $k$.
\end{defn} 

Note that a definable block of positive dimension is a union of connected semi-algebraic sets of positive dimension (the intersection of the definable block with small neighbourhoods of each point), and so, if such a definable block is contained in a set $Z$, it is contained in $Z^{alg}$. \\

In the proof of the main theorem we will need the following refinement of Theorem \ref{Pila_Wilkie} involving the block families:

\begin{thm}[{\cite[Theorem 3.6]{Pila11}}] \label{Pila_blocks}
Let $Z \subset \RR^n \times \RR^m$ be a definable family and $\epsilon \gtr 0$. There is a finite number $J(Z,\epsilon)$ of definable block families $W^{(j)}\subset \RR^n \times (\RR^m \times \RR^{\mu_j})$, $j=1, \ldots, J$, each parametrized by $\RR^m \times \RR^{\mu_j}$ and a constant $c(Z,\epsilon)$ with the following property:
\begin{enumerate}
 \item for each $j$ and $(y,\eta)\in \RR^m \times \RR^{\mu_j}$, $W^{(j)}_{(y,\eta)}\subset Z_y$;
 \item if $Z_y$ is a fiber of $Z$ and $T\ge 1$, the set of rational points of $Z_y$ of height bounded by $T$ is contained in $cT^{\epsilon}$ blocks of the form $W^{(j)}_{(y,\eta)}$ for suitable $j=1, \ldots, J$ and $\eta \in\RR^{\mu_j}$.
\end{enumerate}
\end{thm}

In the previous example, the blocks of Theorem \ref{Pila_blocks} are exactly the vertical lines $z \times [0,1]$ where $z\in \Gamma$ is a rational point of denominator bounded by $T$.

\section{Proof of Theorem \ref{BMZ}}

Let us fix a point $P\in X_{(n-2)}$. By Theorem \ref{bounded_height}, for every (algebraic) point $P\in X_{(n-2)}$, the height $h(P)$ is uniformly bounded by some positive constant $C$. Hence, if we find a bound for the degree $d=[\QQ(P):\QQ]$, we can use Northcott Theorem \cite{Northcott} to achieve finiteness.\\

Let $K$ be a field of definition of the curve $X$ and let us put $d_K=[K:\QQ]$. The set $X_{(n-2)}(\Qbar)$ is stable under Galois conjugation over $K$, hence, if $X_{(n-2)}(\Qbar)$ contains $P$, it must contain all the conjugates of $P$ over $K$ and so at least $[K(P):K]$ points. 

Every $P\in X_{(n-2)}$ lies in some algebraic subgroup $H$
of $\GG^n$ of dimension $r\le n-2$. %We fix such a subgroup $H$. 
Consequently, the coordinates $x_1(P), \ldots, x_n(P)$ satisfy
multiplicative dependencies of the form $x_1(P)^{a_1}\cdots x_n(P)^{a_n}=1$, where the vector $(a_1,\ldots,a_n)$ runs through the lattice $\Lambda_H\subset \ZZ^n$
of rank $n-r \ge 2$ associated to the group $H$. For simplicity of notation, we write 
$\mathbf{x}(P)^{\mathbf a}$ in place of $x_1(P)^{a_1}\cdots x_n(P)^{a_n}$. \\

\noindent In short, here we summarize is the basic strategy of the proof:

\begin{enumerate}
 \item \textbf{Rephrasing the problem:} Consider the exponential map $\exp:\CC^n \rightarrow (\CC^{\times})^n=~\GG^n(\CC)$; a point $P\in \GG^n$ lies in an algebraic subgroup of dimension $r$ if and only if $\exp^{-1}(P)$ is contained in the union of all the translates of $\QQ$-linear subspaces of dimension $r$ by vectors of $2\pi i \ZZ^n$. We choose a fixed determination of the logarithmic function on a suitable compact subset $B$ of $({\CC^{\times}})^n$ such that $B$ contains ``many'' conjugates of $P$ over $K$;
 \item \textbf{Applying a version of the Counting Theorem:} Through the determination of  the logarithmic function fixed before $\mathrm{Log}: B \rightarrow \CC^n\cong \RR^{2n}$, we view points in $X_{(n-2)}\cap B$ as points of $\mathrm{Log}(X \cap B)$ that lie in affine $\QQ$-linear subspaces of $\RR^{2n}$ of a suitable dimension.
Using a result deriving from the third author's strategy (Theorem \ref{final_Pila_gen}), 
we can estimate the number of these points in terms of the height of the linear subspaces;
 \item \textbf{An estimate for the height:} In order to apply the previous step (through Theorem \ref{final_Pila_gen} below), in this part we show that, 
       if $P \cap X_{(n-2)}$, we can always find a $\QQ$-linear subspace
       of $\RR^{2n}$ that contains $\mathrm{Log}(P)$ and with small height with respect to the degree of $P$ over $\QQ$. 
       This part uses some standard facts of Diophantine geometry; 
 \item \textbf{Conclusion of the argument:} Using the steps 2 and 3 and the fact that the set $X_{(n-2)}(\Qbar)$ is stable under Galois conjugation over $K$,
       we obtain the wanted bound the for degree of $P$. Combining this bound with the bound for the height of $P$ we have the thesis. \\
\end{enumerate} 
\vspace{0.2 cm}
In more details, we can proceed along the following steps. We shall be very brief in certain details, especially those which appear
elsewhere. \\ 

\noindent \textbf{4.1 Rephrasing the problem:} \\

Let us consider the usual exponential map $\exp: \CC^n \rightarrow (\CC^{\times})^n=\GG^n(\CC)$.
A point $P\in \GG^n$ lies in an algebraic subgroup of dimension $r$ if and only if its coordinates satisfy $n-r$ conditions of multiplicative 
dependence, so if and only if $\exp^{-1}(P)$ is contained in the union of all the translates of $\QQ$-linear subspaces of dimension $r$ by vectors 
of $2\pi i \ZZ^n$. \\

We shall use logarithms in what follows. To start with, we construct a certain simply connected set $B\subset \GG^n(\CC)$; it shall be convenient,
for reasons we are going to explain later, to choose a $B$ that contains a positive percentage of conjugates of $P$. 
For this, let us first define, for $0\less \delta \less 1$, the set:
\begin{equation} \label{definition_Tdelta}
 T_{\delta}= \left \{ P\in \GG^n\ \ |\ \  \delta \le \left | x_i(P) \right |\le \frac{1}{\delta} \ \ \forall\ i=1, \ldots, n  \right \}. 
\end{equation}
We want to choose a suitable small $\delta$ such that, if $P\in T_{\delta}$, at least one half of the conjugates of $P$ (over $K$) lie in $T_{\delta}$.
The existence of such a $\delta$ is guaranteed by the following result:
\begin{prop} \label{delta}
Given a number field $K$ and a constant $C$, there exists a positive constant $\delta=\delta(C,K)\less 1$, with the following property:
if $P\in \GG^n(\Qbar)$ with $h(P)\le C$, there are at least $\frac{1}{2}[K(P):K]$ different embeddings $\sigma$ of $K(P)$ in $\CC$ such that 
$\sigma(P)$ lies in $T_{\delta}$. 
\end{prop}

This is easily proved: indeed, if $\frac{1}{2n}[K(P):K]$ conjugates of $P$ over K fail to sati\-sfy one of the conditions 
defining $ T_{\delta}$, for instance $|x_1(P)|\le \frac{1}{\delta}$, then the height of $x_1(P)$ would be $\gg \log\left ( \frac{1}{\delta} \right )$,
contradicting for sufficiently small $\delta$ the bound for the height given by Theorem \ref{bounded_height}. A detailed proof of this proposition in a 
similar context can be found in \cite[Lemma 8.2]{Masser_Zannier_2012}. \\

Now, $T_{\delta}$ is compact but not simply connected; however, we can cover it with $s$ simply connected sets for a sufficiently large 
$s$. We may suppose that each of them is a product of  closed disks. By the Pigeon-hole Principle, 
one of these sets will contain at least $\frac{[K(P):K]}{2s}$
conjugates of $P$ over $K$. We will fix this set and call it $B$. \\

As $B$ is a product of simply connected sets, 
we can define the map $\mathrm{Log}: B \rightarrow \CC^n$ given by $(x_1, \ldots, x_n) \mapsto (\log x_1, \ldots, \log x_n)$,
where $\log$ denotes the principal determination of complex logarithm, namely $\log(\rho e^{2\pi i \theta})=\log \rho+ 2\pi i \theta$ with $\rho \in \RR^+$ and $\theta \in [0,1)$.\\ 

Consider now the set \[ \mathcal Z:= \mathrm{Log}(X\cap B)=\left \{ \left ( \,\log(x_1(P)), \ldots, \log(x_n(P))\, \right )\in \CC^n \ | \ P\in X\cap B \right \}; \] 
as observed above, the intersections of $X$ with algebraic subgroups of $\GG^n$ of dimension $r$ correspond to intersections between $\mathcal Z$ and
translates of $\QQ$-linear subspaces of $\CC^n$ by vectors of $2\pi i \ZZ^n$ of the same dimension.
More specifically, if $P\in X\cap B$ satisfies a multiplicative dependence condition of the form $\mathbf x(P)^{\mathbf a}=1$, 
its image $\mathrm{Log}(P)$ satisfies a linear condition of the form
\begin{equation} \label{complex_relation}
a_1 \log(x_1(P))+ \cdots + a_n \log(x_n(P))=2\pi i b_{\mathbf a}(P) \end{equation}
for a certain integer $b_{\mathbf a}(P)$ depending on the vector $\mathbf a$ and on the point $P$. \\

\noindent \textbf{4.2 Applying a version of the Counting Theorem:} \\

We shall apply Theorem \ref{Pila_blocks} to have an estimate for the points $P\in X \cap B$ which satisfy linear conditions of the form (\ref{complex_relation}) in terms of the maximum of the absolute values of the coefficients $a_1, \ldots, a_n.$ To do this, 
we need to consider $\mathcal Z$ as a real variety. 

We can identify $\CC^n$ with $\RR^{2n}$ via 
the isomorphism $\varphi:\CC^n \rightarrow \RR^{2n}$ defined by 
\[(y_1, \ldots, y_n) \longmapsto  (\Real(y_1),\ \Imag(y_1)/2\pi, \ldots, \Real(y_n),\ \Imag(y_n)/2\pi ). \]

Let us define $Z:=\varphi(\mathcal Z)=\varphi(\mathrm{Log}(X \cap B))$;
it is clear that $Z$ is a real compact subanalytic surface of $\RR^{2n}$; in particular, it is definable in the o-minimal structure $\RR_{an}$ 
(for this notion, see for instance \cite{van_den_Dries} or \cite{Zannier12}). \\

A point $P\in X\cap B$ lies in an algebraic subgroup $H$ of dimension $r$ if and only if it satisfies
multiplicative dependence conditions $\mathbf x(P)^{\mathbf a}=1$, where $\mathbf a$ runs through the associated lattice 
$\Lambda_H\subset \ZZ^n$ of rank $n-r$. 
Taking real and imaginary parts of the relation (\ref{complex_relation}), this happens if and only if its image 
$P'=\varphi(\mathrm{Log}(P))$, having coordinates denoted by $(z_1(P'), \ldots, z_{2n}(P'))$, satisfies linear conditions of the form 
\begin{equation} \label{generalized_system_form} 
\begin{cases} a_{1} z_1 +a_2 z_3 + \cdots + a_{n} z_{2n-1}=0 & \\  
        a_{1} z_2+ a_2 z_4 + \cdots +a_{n} z_{2n}=  b_{\mathbf a}(P) \end{cases}  
\end{equation}
where  $\mathbf a=(a_1,\ldots,a_n)$ runs through $\Lambda_H\subset \ZZ^n$ and $b_{\mathbf a}(P)$ is an integer depending on the vector $\mathbf a$ and on 
the point $P$. 
For a fixed lattice $\Lambda_H$ and fixed $P$, the system above describes a $\QQ$-affine subspace of $\RR^{2n}$ of dimension $2r$. \\

In order to deal with $\QQ$-linear subspaces instead of $\QQ$-affine ones,
we can identify $\RR^{2n}$ with the points of the projective space $\PP_{2n}(\RR)$ of coordinates $[z_1 : \cdots : z_{2n} : 1]$. 

In this way, a system of the form 
\begin{equation} \label{projective_system_form} 
\begin{cases}
 a_{1} z_1 +a_2 z_3 + \cdots + a_{n} z_{2n-1}=0 & \\  
        a_{1} z_2+ a_2 z_4 + \cdots +a_{n} z_{2n}- b_{\mathbf a}z_{2n+1}=0 \end{cases}  
\end{equation}
where  $\mathbf{a}=(a_1,\ldots,a_n)$ varies in a lattice of rank $n-r$ and $b_{\mathbf a}$ is an integer, identifies a 
$\QQ-$linear subspace of $\PP_{2n}(\RR)$ of dimension $2r$. We denote by $Gr_{2r}^{2n}$ the Grassmannian variety of all the 
linear subspaces of $\PP_{2n}(\RR)$ of dimension $2r$. \\

We are then interested in counting points of $Z$ that lie in some linear subspace of $\PP_{2n}(\RR)$ of dimension $2r$ of the form (\ref{projective_system_form}). 

\begin{rmk}
We notice that, if a point $P\in \GG^n$ lies in a certain algebraic subgroup $H$, 
all the conjugates $\sigma(P)$ of $P$ will lie in the same algebraic subgroup $H$. However, this does not remain strictly true when we consider, in place of $\sigma(P)$, the image through the map $\varphi \circ \mathrm{Log}$. Indeed, the images of the conjugates $\sigma(P)\in B$ of through $\varphi \circ \mathrm{Log}$ do not necessarily lie in a unique $\QQ$-affine subspace of $\RR^{2n}$ of the form 
(\ref{generalized_system_form}): this is because the integer $b_{\mathbf{a}}$ depends on the point so it may change if $P$ is replaced by one of these conjugates. 

On the other hand, we point out that a bound for the cardinality of the points in question which lie in the same affine subspace $L$, i.e., for those with the same 
integer $b_{\mathbf a}$, is a consequence of Gabrielov's Theorem \cite{Gabrielov}. 
\end{rmk}

To have an estimate for the points of $Z$ which lie in some linear subspace of $\PP_{2n}(\RR)$ of dimension $2r$ of the form (\ref{projective_system_form}), 
we will adapt the proof of the general Theorem \ref{Pila-thm}, proved in the Appendix, to our particular case. This will be done in Theorem 
\ref{final_Pila_gen}.

Roughly speaking, Theorem \ref{Pila-thm} gives estimates for the number of points in a ``transcendental variety'' 
which lie in linear subspaces of a given dimension with rational coefficients and bounded height. 

It turns out that these points are very ``few'' outside some ``anomalous part'' given by the union of all 
the connected components of positive dimension of the intersections between the transcendental variety and algebraic sets of a certain
appropriate dimension. \\

Let us denote by $Z^{\langle 2r \rangle}$ the set of the points of $Z$ that 
lie in a $\QQ$-linear subspace of $\PP_{2n}(\RR)$ of dimension $2r$ of the form (\ref{projective_system_form}).  
We will call these points $\langle 2r \rangle$-points of $Z$. \\

On Grassmannian varieties one can define a suitable notion of height, as in \cite[Section 2.8, p.\,83]{Bombieri_Gubler}. 
With this notion, if $L\in Gr_M^{N}$ is a rational point defined by a system of the form $ A \mathbf z= \mathbf{0}$,
where $A$ is a $(N-M) \times (N+1)$-matrix with integer coefficients, the height of $L$ is given by
\[ H(L)=D^{-1}\sqrt{|\det(AA^{t})|}, \]
where $D$ is the greatest common divisor of the determinants of all the $(N-M) \times (N-M)$ minors of $A$. 
This height can be easily estimated by \begin{equation} \label{height_matrix}
H(L)\le (\sqrt{N+1} \max_{i,j} |a_{ij}|)^{N-M}. \end{equation}

\begin{defn} 
We define the height of a $\langle 2r \rangle$-point as
the minimum height of a rational element of $Gr_{2r}^{2n}$ of the form (\ref{projective_system_form}) that contains the point. 
For $T\ge 1$, we denote by $Z\langle 2r, T\rangle$
the set of $\langle 2r \rangle$-points of $Z$ with height bounded by $T$ and by $N^{\langle 2r \rangle}(Z,T)$ its cardinality.  
\end{defn}
%\vspace{0.3 cm} 

%Next theorem gives an estimate for the cardinality of the number of points 

\begin{thm} \label{final_Pila_gen}
Let $Z\subset \PP_{2n}(\RR)$ be the real subanalytic set defined above and let $r$ be a positive integer $\le n-2$. For every $\epsilon\gtr 0$, there exists a constant $c(Z,n,\epsilon)$ such that 
\[ N^{\angr}(Z,T) \le c(Z,n,\epsilon)T^{\epsilon}, \] 
i.e., the number of points of $Z$ that lies in a $\QQ$-linear subspace of dimension $2r$ of $\PP_{2n}(\RR)$ of the form (\ref{projective_system_form})
with height bounded by $T$ is less than $c(Z,n, \epsilon) T^{\epsilon}$.\footnote{As remarked also in the Appendix, this result is uniform 
for families of definable sets.}
\end{thm}

%\vspace{0.2 cm}
\begin{proof}
The proof of this theorem follows the structure of the proof of Theorem \ref{Pila-thm} stated in the Appendix. \\

If $L$ is a point of $Gr_{2r}^{2n}$, we denote by $S(L)$ the corresponding linear subvariety in $\PP_{2n}(\RR)$. We will denote by $c_1,c_2, \ldots$ some 
positive constants. 

Let us define the set 
\[ W_{Z}=\{ L\in Gr_{2r}^{2n}\ |\ S(L) \mbox{ is of the form } (\ref{projective_system_form}) \mbox{ and } S(L) \cap Z \neq \emptyset \}. \]
This set is a definable subset of $Gr_{2r}^{2n}$ and, by definition, a $\langle 2r \rangle$-point of $Z$ lies in the intersection 
of $Z$ with a rational point of $W_Z$. \\

As $Gr_{2r}^{2n}$ is a projective variety, we consider it to be the union of finitely many affine pieces. 
So, without loss of generality, we may suppose that $W_{Z}$ is contained in one of these pieces. \\

Given $\epsilon >0$, we can apply Theorem \ref{Pila_blocks} to conclude that the set of rational points of $W_Z$ with height bounded by $T$, denoted by $W_Z(\QQ,T)$, is contained in 
$c_1(Z,n, \epsilon)T^{\epsilon/2}$ blocks from some finite collection of block families. As explained in Definition \ref{basic_block}, these blocks are connected definable locally semialgebraic sets 
(some possibly of dimension 0) contained in $W_Z(\QQ,T)$, with some properties of regularity.\\

Consider now a block $U$ of $W_Z$ of dimension $m$. We want to prove that each block intersects $Z$ 
in a finite number of points bounded independently from the block,\footnote{Here it is crucial to apply the uniform version of Theorem \ref{Pila_blocks} for families of definable sets.} i.e. that
$|\bigcup_{L\in U} S(L)\cap Z|\le c_2(Z,n,\epsilon)T^{\epsilon/2}. $ \\

\noindent To prove this inequality, we use an inductive argument on the dimension $m$ of the block. \\
\begin{itemize}
 \item If $m=0$, the block contains only a point $L$. Firstly, let us prove that, if $L\in W_Z(\QQ)$, then $S(L)\cap Z$ is finite. 
Suppose in fact by contradiction that there exists $L\in W_Z$ such that $S(L)\cap Z$ is infinite.
%Since $Z$ is compact, $S(L)\cap Z$ is still compact, 
%hence it has positive dimension, and it contains a one dimensional real analytic curve segment whose points lie on a fixed linear variety of the form
%(\ref{projective_system_form}). 
This gives an infinite set of points $\mathcal S \subseteq X \cap B$ such that a relation of the form
\[ m_1 \log x_1(P)+ \cdots + m_n\log x_n(P)=2\pi i m_{n+1}  \]
holds identically on $\mathcal S$ for $m_1, \ldots m_{n+1}\in \QQ$. As $X\cap B$ is compact, $\mathcal S$ has an accumulation point in $X \cap B$, hence by standard principles for analytic functions (``Identity Theorem'' or 
\cite[Thm.\,1.2, p.\,90]{Lang_complex}), this relation holds on all $X \cap B$, contradicting the hypothesis that $X$ is not contained in any proper algebraic 
subgroup. \\

Hence, each intersection of $S(L)$ with $Z$ for $L \in W_Z(\QQ,T)$ is finite and, as the 
number of connected components of the fibers of a definable family in an o-minimal structure is uniformly bounded,
of cardinality uniformly bounded by some positive constant $c_3(Z,n)$.\footnote{In our setting, as the set $Z$ is subanalytic, this uniformity
for the cardinality of the intersections is, as remarked before, a consequence of Gabrielov's Theorem.} \\

\item Suppose now $m\gtr 0$ and fix a block $U$ of dimension $m$. There may be some points $P\in Z$ that lie on $S(L)$ for every $L\in U$; let us call them $\{P_i\}_{i=1}^k$. These points will number at most $ c_3$. We prove that the subset $V \subseteq U$ of points $L\in U$ such that the 
       corresponding linear subspace of $\PP_{2n}(\RR)$ intersects $Z$ in another point not among the fixed $P_i$ has dimension $\less m$.
Moreover, the set $V$ can be taken from a definable
family for all the blocks $U$ in its definable family (hence the Counting Theorem can be applied to them with uniform constants).   \\
       
Suppose by contradiction that this is false; hence, there exists some subset $U' \subseteq U$ of full dimension on which we can construct a
       definable function $f:U' \rightarrow Z$ taking an element $L$ to a point of $S(L)\cap Z$ not among the $P_i$. 
       We can also suppose that $f$ is differentiable on some full dimensional subset $U''\subseteq U'$ of regular points of $U'$. 
       Let $Q$ be a point of $U''$ and suppose that the derivative of $f$ is non-zero in some direction.\\
       Then, intersecting $U''$ with a suitable linear subvariety, we have an algebraic curve segment $\mathcal C\subseteq U''$ such that $f(\mathcal C)\subseteq Z$ 
       is non constant. We can also restrict $\mathcal C$ such that $f_{\mid \mathcal C}$ is invertible.\\
       So, if $L \in \mathcal C \subseteq U''$, the corresponding subvariety $S(L)$ is a linear affine space of $\PP_{2n}(\RR)$ 
       defined by a system of the form 
       \begin{equation} \label{system_form} 
       \begin{cases} a_{k\,1} z_1 + a_{k\,2} z_3+  \cdots +a_{k\,n} z_{2n-1}=0 & \\  
                     a_{k\,1} z_2+ a_{k\,2} z_4 +\cdots +a_{k\,n} z_{2n}=  b_kz_{2n+1} \end{cases} \mbox{ for } k=1, \ldots, n-r, \end{equation}
       and the transcendence degree of $F=\CC(a_{ij},\, b_i, i=1,\ldots, n-r, \ j=1,\ldots, n)/\CC$ restricted to the algebraic curve $\mathcal C$ is 
       equal to 1. \\

       Now, consider $f(\mathcal C) \subseteq Z$; as $f$ is not constant along $\mathcal C$, the set $f(\mathcal C)$ contains an infinite subset $E\subseteq Z$ and,
       as by definition $Z=\varphi(\mathrm{Log}(X \cap B))$, the image of an infinite set $E'\subseteq~X\cap B$. 
       Moreover, for every $P\in E'$, the image $\varphi(\mathrm{Log}(P))$ of $P$ of coordinates $[z_1(P): \cdots : z_{2n}(P): 1]$ satisfies a system of the form 
       (\ref{system_form}). 
       If we multiply the second relation by $2\pi i$ and we sum it to the first one for every $k$, we obtain 
       $n-r$ independent relations of the form 
       \[ a_{k\,1} \log(x_1(P)) + \cdots + a_{k\,n} \log(x_n(P))=2\pi i b_k \quad \mbox{ for } k=1, \ldots, n-r. \]
       Hence, the transcendence degree of $F(\,\log(x_1(P)), \ldots, \log(x_n(P))\,)/F$ on $E'$ is $\le r$. \\
       From this, we have that the transcendence degree of $\CC\left (\log(x_1(P)), \ldots, \log(x_n(P))\right )/\CC$ on the infinite set $E'$ is $\le r+1\le n-1$ as, by assumption, $r\le n-2$\footnote{Here it is crucial the assumption $r\le n-2$ otherwise the result is false.}. 
       
This means that, on $E' \subset X \cap B$, the functions $\log(x_1(P)), \ldots, \log(x_n(P))$ are algebraically dependent over $\CC$. 
       Again using the ``Identity Theorem'', they remain dependent on all $X\cap B$, and this is impossible applying Ax-Schanuel's Theorem for functions (see \cite{Ax}).\footnote{Notice that, to apply \cite{Ax} we need the assumption that the curve $X$ is not contained in any translate of a proper algebraic subgroup.} Hence, $f$ is constant, giving a new point lying on all the $S(L)$ for every $L\in U$ that is
       a contradiction. \\
       
       This means that the set $V\subseteq U$ of $L$ that intersects $Z$ in another point different from the $P_i$ has dimension $\less m$.
       Applying again Theorem \ref{Pila_blocks} on $V$ and the induction hypothesis on the blocks of $V$ (that will be of dimension $\less m$), we have the wanted bound. \\
\end{itemize}

\noindent We can finally put the two estimates together to prove the assertion:
\[ N^{\langle 2r \rangle}(Z,T) \le |\{\mbox{blocks}\}|\cdot \left \lvert \bigcup_{L\in U} S(L)\cap Z \right \rvert \le c_4(Z,n,\epsilon) T^{\epsilon}, \]
as wanted.
\end{proof}

\vspace{0.5 cm}

\noindent \textbf{4.3 An estimate for the height:} \\

In order to apply Theorem \ref{final_Pila_gen}, we need to find a bound for the height of the $\angr$-points of $Z$ in terms of the degree of $P$. With this aim, in this 
section we show that, for every $P\in X \cap B$ that lies in an algebraic subgroup $H$ of dimension $r$, 
we can always construct a suitable linear subspace $L'\subseteq \PP_{2n}(\RR)$ of the form (\ref{projective_system_form}) that contains $P'=\varphi(\mathrm{Log}(P))$ 
and with small height with respect to the degree of $P$ over $\QQ$. \\

This part of the argument follows some ideas of the original proof of \cite{Bombieri_Masser_Zannier_1999}, 
although this method allow us to use some weaker results. As already remarked before, 
these simplifications may result to be useful in other contexts. \\

Let us fix $P\in X \cap B$ such a point and choose $r$ as the maximal dimension such that $P$ lies in an algebraic subgroup of
$\GG^n$ of dimension $r$. 
Let $\Gamma \subset \overline \QQ ^{\times}$ be the subgroup generated by the coordinates $x_i(P)$; the rank of $\Gamma$ is equal to $r$. 
We can choose some generators of the subgroup with a good behaviour with respect to the logarithmic Weil height, thanks to the following lemma:
\begin{lem}[{\cite[Lemma 2]{Bombieri_Masser_Zannier_1999}}] \label{generators_of_Gamma}
Let $\Gamma$ be a finitely generated subgroup of $\Qbar^{\times}$ of rank $r$. There exist elements $g_1, \ldots, g_r \in \Gamma$ 
generating $\Gamma/\mbox{tors}$ and a positive constant $C(r)$ such that 
\begin{equation}
h(g_1^{k_1} \cdots \, g_r^{k_r}) \ge C(r) (|k_1|h(g_1)+ \cdots + |k_r| h(g_r)),
\end{equation}
for every $k=(k_1, \ldots, k_r) \in \ZZ^r$. The constant $C(r)$ is explicit and can be taken equal to $r^{-1}4^{-r}$. 
\end{lem}
\noindent This lemma 
is an immediate consequence of \cite[Thm.\,1.1]{Schlickewei}. \\

Using the generators $\{g_1, \ldots, g_r\}$ given by the previous lemma, let us write the coordinates of $P$ as
\begin{equation}
x_i(P)=\zeta_i g_1^{t_{i\, 1}} \cdots g_r^{t_{i\, r}} \quad \mbox{for every } i=1, \ldots, n, 
\end{equation}
where $\zeta_i$ is some root of unity.
Then, applying Lemma \ref{generators_of_Gamma} and the bound for the height of $P$, we have that
\[ 1\gg h(P) \gg h(x_i(P)) \gg |t_{i1}|h(g_1)+ \cdots + |t_{ir}| h(g_r) \qquad \mbox{for every } i=1, \ldots,n.  \]
Consequently, if  we call $T_j=\max_i{|t_{ij}|}$,
we have
\begin{equation} \label{h(g_j)}
h(g_j) \ll \frac{1}{T_j} \quad \mbox{for every } j=1, \ldots, r. \end{equation}

We write the torsion parts as powers of a suitable primitive $N$-root of unity 
$\zeta_N$, i.e. $\zeta_i=\zeta_N^{d_i}$ with $0\le d_i\less N$ for every $i=1, \ldots, n$. If $N$ is the smallest possible common order, we have that
$\zeta_N \in \QQ(P)$, hence, denoting by $\phi$ the Euler function, we have that $\phi(N) \ll d$. Moreover, it is well known that
$\phi(N) \gg \sqrt N$, so we have that $N\ll d^2$. \\

Now, a multiplicative dependence condition of the form $\mathbf x(P)^{\mathbf c}=1$ for some vector 
$ \mathbf c=~(c_1, \ldots, c_n) \in \ZZ^{n}$ gives
\[ 1=\mathbf x(P)^{\mathbf c}= \prod_{i=1}^n \zeta_N^{c_i d_i} \prod_{j=1}^r g_j^{\sum_{i=1}^n c_i t_{ij}}. \]

Taking the conditions on the exponents, we obtain a linear system of $(r+1)$ equations with integer coefficients and $n$ unknowns 
$\{c_1, \ldots, c_n\}$ of the form
\begin{equation} \label{system_1}
\begin{cases} \sum_{i=1}^n c_i t_{ij}=0 &  \qquad \mbox {for } j=1, \ldots, r,\\
                 \sum_{i=1}^n c_i d_i\equiv 0 \mod N. 
                 \end{cases} 
\end{equation}
We can easily find ``small'' solutions to the first $r$ linear equations $\sum_{i=1}^n c_i t_{ij}$ for $j=1,\ldots, r$ in the $n$ variables $c_i$ using elementary linear algebra. In fact, suppose without loss of generality that the first minor $r \times r$ of the matrix $(t_{ij})$ is not zero. If $T= \max_{ij} |t_{ij}|$, its determinant is easily bounded in absolute value by $r!\, T^r$. Hence, applying for $j=1, \ldots, r$ Cramer method to the system $\sum_{i=1}^r c_i t_{ij}=-\sum_{i=r+1}^n c_i t_{ij}$ in the unknowns $\{t_{1j}, \ldots, t_{rj}\}$, we can find $n-r$ linearly independent vectors $\mathbf{l_k}~=(l_{k1}, \ldots, l_{kn})\in \ZZ^n$ for $k=1, \ldots, n-r$ such that 
they satisfy the first $r$ equations of the system (\ref{system_1}) and with
\[ \max_{k,s} |l_{ks}| \ll T^r. \] 

Then, multiplying all these vectors by $N$, we obtain $n-r$ linearly independent vectors 
$\mathbf{l'_k}=~(l'_{k1}, \ldots, l'_{kn})\in \ZZ^n$ for $k=1, \ldots, n-r$
which satisfy the above system (\ref{system_1}) and with 
\begin{equation} \label{l'}
\max_{k,s} |l'_{ks}| \ll N T^r \ll d^2 T^r. \end{equation}
If $ j'\in \{1, \ldots, r\}$ is the index such that $T=T_{j'}= \max_j T_{j} $, using (\ref{h(g_j)}) we have that $T \ll \frac{1}{h(g_{j'})}$ and,
applying Blanksby and Montgomery's result (\ref{BM}) to $g_{j'}\in \QQ(P)$, we have
$h(g_{j'})\gg d^{-3}$, hence $T \ll d^3$. 
Putting this into (\ref{l'}), we have
\[ \max_{k,s} |l'_{ks}|\ll d^{3r+2}. \]

Hence, if $\Lambda_{H'}$ is the lattice of $\ZZ^{n}$ generated by the so constructed $n-r$ vectors $\mathbf{l'_k}\in \ZZ^n$,
the associated algebraic subgroup $H'$ of $\GG^n$ has dimension $r$ and contains the point $P$ (and all its conjugates). \\

Now, as seen in the previous section, if $P\in X\cap B$ lies in this algebraic subgroup $H'$,
the image $P'=\varphi(\mathrm{Log}(P))\in \PP_{2n}(\RR)$ of coordinates
$[z_1(P'): \ldots : z_{2n}(P'): 1]$ lies in a $\QQ$-linear subspace $L'$ of $\PP_{2n}(\RR)$ of dimension $2r$ of the form 
\begin{equation} \label{generalized_system_form_2} 
\begin{cases} l'_{k\,1} z_1 + l'_{k\,2} z_3+  \cdots + l'_{k\,n} z_{2n-1}=0 & \\  
l'_{k\,1} z_2+ l'_{k\,2} z_4 +\cdots +l'_{k\,n} z_{2n}=  b_{\mathbf {l'_k}}(P) z_{2n+1} \end{cases} \mbox{ for } k=1, \ldots, n-r. \end{equation}
But, as we considered the principal determination of the logarithmic function, the even coordinates of $P'$ are all in $[0,1)$ and $z_{2n+1}(P')=1$, 
hence we can easily estimate $|b_{\mathbf {l'_k}}(P)|$ by
\begin{equation} \label{h_b}
 |b_{\mathbf {l'_k}}(P)|\le n \max_{ks} |l'_{ks}|\ll d^{3r+2}.
\end{equation}
As already remarked before, notice that, if $\sigma(P)$ is a conjugate of $P$ over $K$ lying in $B$, the point $\varphi(\mathrm{Log}(\sigma(P)))$
 may lie in a different $\QQ$-linear subspace of $\PP_{2n}(\RR)$ (because the coefficient $b_{\mathbf {l'_k}}(P)$ depends on $P$), but
 with height bounded by the same quantity. \\

Therefore, for every $P\in X_{(n-2)}\cap B$, we constructed a $\QQ$-linear subspace $L'$ of $\PP_{2n}(\RR)$ of the form (\ref{projective_system_form}) 
such that it contains the point $P'=\varphi(\mathrm{Log}(P))$ and, combining (\ref{height_matrix}) and (\ref{h_b}), with height bounded by
\begin{equation} \label{height L'}
H(L') \ll d^{2(3r+2)(n-r)}.  
\end{equation}

%\newpage
\vspace{0.4 cm}

\noindent \textbf{4.4 Conclusion of the argument:} \\

We can finally put all the results together to conclude the argument. We will denote by $\gamma_1,\ \gamma_2, \ldots$ some positive constants.\\
\begin{itemize}
 \item Let us take the point $P$ with coordinates satisfying the condition (\ref{2-condition}); applying Theorem~\ref{bounded_height}, its height is uniformly bounded. Set $d=[\QQ(P):\QQ]$. We can choose a suitable set $B\subseteq \GG^n(\CC)$ where we can take the principal determination of the logarithmic function $\mathrm{Log}: B \rightarrow \CC^n$ and
which contains at least $\gamma_1\cdot d$ conjugates of $P$ over $K$ for a certain positive constant $\gamma_1$. As these conjugates lie in the set $X_{(n-2)}$, we get $|X_{(n-2)} \cap B|\ge \gamma_1 \cdot d$.\\
 \item Identifying $\CC^n$ with $\RR^{2n}\subset \PP_{2n}(\RR)$ via the isomorphism $\varphi$ defined before and taking the real surface $Z=\varphi(\mathrm{Log}(X\cap B))$, 
these $\gamma_1\cdot d$ points of $X_{(n-2)}\cap B$ give rise to the same number of points of $Z$ that lie in 
       some $\QQ$-linear subspaces of $\PP_{2n}(\RR)$ 
       of dimension $2r$ associated to a linear system of the form (\ref{projective_system_form}). 
       In Section 3, we showed that each of these points of $Z$ lies in a $\QQ$-linear subspace of $\PP_{2n}(\RR)$ of the said form and with height bounded by $\gamma_2 d^{2(3r+2)(n-r)}$ for some positive constant $\gamma_2$. \\
 \item We can apply Theorem \ref{final_Pila_gen} with $T=\gamma_2 d^{2(3r+2)(n-r)}$; so, for every $\epsilon \gtr 0$,
       \begin{equation} 
       \gamma_1 d \le |X_{(n-2)} \cap B| \le N^{\angr}(Z,\gamma_2 d^{2(3r+2)(n-r)}) \le \gamma_3 d^{ 2(3r+2)(n-r) \epsilon}. \end{equation}
       
       Finally, taking a suitable value of $\epsilon$, e.g. $\epsilon=\frac{1}{4(3r+2)(n-r)}$, we obtain a bound for the degree of $P$.
       Combining this with the bound for the height and using Northcott Theorem \cite{Northcott},
       we have the wanted result of finiteness, proving the theorem. 
\end{itemize}

\newpage
\section*{Appendix}

This appendix is drawn from some informal notes \cite{Pila}
which had some very limited circulation. 
The basic objective was to show that the Counting Theorem for
rational points on definable sets \cite{Pila_Wilkie} could be adapted to study
``unlikely intersections'' in a definable context. A further objective was 
to make the result ``look like'' the Diophantine problems of
``unlikely intersections'' and to study the
exceptional set in the Diophantine settings on the assumption 
of suitable functional transcendence properties. 
These last aspects we omit here as they have been carried further,
in somewhat different formulations, by Habegger and the third author in \cite{HP12}, \cite{HP14},
to obtain results (in general conditional) on ``unlikely
intersections'' for subvarieties of abelian varieties and 
products of modular curves.  

Let $Z\subset\RR^n$ be a definable set of dimension $k$.
We will consider points of $Z$ whose coordinates satisfy some 
given number of independent linear conditions over $\QQ$.
As will be clear from the proofs, the main result is uniform for definable
families, but we state things in terms of a single set $Z$ for simplicity.

\begin{defn}
For integers $\lambda\ge 0$ and $ n\ge 1$ we denote by 
${\rm Gr}^n_\lambda$ the Grassmann variety of 
$\lambda$-dimensional linear subvarieties of $\PP_n(\RR)$. 
\end{defn}

\begin{defn}
A {\it $\langle\lambda\rangle$-point\/} in $\RR^n$ is a point that lies on an element of ${\rm Gr}_\lambda^n$ whose coordinates  are rational. 
The {\it height\/} of a $\langle\lambda\rangle$-point
is the minimum height of an element in the Grassmannian variety containing the point.
\end{defn}

For the modular settings one needs analogous notions in a product
$\HH^n$ of upper half-planes. Here the ``linear'' subvarieties are
subvarieties defined by equations of the form $z_i=gz_j$ where
$g\in{\rm SL}_2(\RR)$ acts as M\"obius transformations.
One must also deal with ``algebraic points of bounded degree''.
These modifications are all straightforward and we omit the details.
The same framework works for any mixed Shimura variety.

We introduce next a variant of the notion of a ``block'' used in \cite{Pila09} (see also \cite{Pila11}) and recalled in Definition \ref{basic_block} to give some refined versions of the Counting Theorem.

%\medskip
\begin{defn}
A {\it $(\kappa, \nu, \beta)$-block\/} in $\RR^n$ is a connected definable set $B\subset\RR^n$, of dimension $\kappa$, regular (of dimension $\kappa$) 
at every point, contained in some irreducible closed algebraic set $A$ of dimension $\nu$ and degree
$\le \beta$. We call $\kappa$ the {\it dimension\/} of the block, we call
$\beta$ the {\it degree\/} of the block, and we call
$\delta=\nu-\kappa$ the {\it deficiency\/} of the block. 
\end{defn}

Dimension $0$ is allowed, though sometimes we need to specify positive dimension. We introduce a further notation emphasising the deficiency rather than the two dimensions. 

\begin{defn} \label{blocks}
A {\it $[\delta, \beta]$-block\/} is a block
of degree $\beta$ and deficiency $\delta$.
\end{defn}

\begin{defn}
Let $Z\subset\RR^n$. For integers $\delta\ge 0$ and $\beta\ge 1$,  the {\it $[\delta, \beta]$-part\/}  of $Z$ is the union of all the $[\delta, \beta]$-blocks of positive dimension contained in $Z$, and is denoted $Z^{[\delta, \beta]}$.
\end{defn}

For example, consider $Z\subset\RR^n$ of dimension $k$. 
Then the subset of regular points of $Z$ of dimension $k$ is definable, 
and has finitely many connected components. 
These ``big blocks'' in $Z$ are then $(k,n,1)$-blocks, 
and are contained in the $[n-k,1]$-part of $Z$.
The $[\delta,\beta]$-parts are supposed to be analogues 
of the ``anomalous'' parts in the Diophantine settings
(see e.g. \cite{Bombieri_Masser_Zannier_2008}).
The $[0,\beta]$-part of $Z$ is the union of  
``definable blocks'' of $Z$ (as defined in Definition \ref{basic_block})
of degree $\le \beta$, 
and so is a definable subset of the ``algebraic part'' $Z^{\rm alg}$
(see for example \cite{Pila_Wilkie}).

\begin{defn}
Suppose $Z\subset\RR^n$. We let
\[ Z^{\langle\lambda\rangle} \]
denote the set of $\langle\lambda\rangle$-points of $Z$ and, 
for  $T\ge 1$,
\[ Z \langle \lambda, T\rangle \]
the set of $\langle\lambda \rangle$-points of $Z$ of height $\le T$,
and finally
\[ N^{\langle\lambda\rangle}(Z,T)= Z\langle\lambda, T\rangle. \]
\end{defn}

\noindent The following is then the basic result for unlikely intersections of definable sets.

\begin{thm} \label{Pila-thm}
Let $Z\subset\RR^n$ be a definable set of dimension $k$. Let
$\lambda\ge 0$ be an integer and $\epsilon>0$. There is an integer 
$b=b(k,n,\lambda,\epsilon)\ge 1$ 
and a constant $c=c(Z,\lambda, \epsilon)$ such that
\[ N^{\langle\lambda\rangle}(Z-Z^{[\lambda,b]},T) \le c\, T^\epsilon. \]
\end{thm}

\begin{rmk} \ \\
 
\begin{itemize}
\item Observe that $b$ is independent of $Z$ and depends only on the numerical data
(and $k$ is bounded by $n$ so could be suppressed, 
but the $b$ value will be smaller if $k$ is smaller).
\item The example $z=x^y$, with $x,y\in (2,3)$ say,
which contains blocks of rational curves of arbitrarily high degree, shows that in general we need to have $b(\epsilon)\rightarrow\infty$ as 
$\epsilon\rightarrow 0$.
\item If $\lambda=n-k$ (or bigger) then the theorem is almost 
vacuous since the ``big blocks'' 
of $Z$ are in the $[\lambda, 1]$-part. 
But it is not completely vacuous as there may be 
points of $Z$ that are not regular of dimension $k$.
\item When $Z$ arises from a Diophantine setting, one hopes to
characterise $Z^{[\lambda, b]}$ in suitable precise terms 
(functional transcendence; in particular $b$ will be bounded) 
and improve $\ll_\epsilon~T^\epsilon$ to finite
(via lower bounds for Galois orbits).
This is carried out (at least conditionally) in the abelian and
modular settings in \cite{HP12}, \cite{HP14}.
\end{itemize}
\end{rmk}

For a subset $Z\subset\RR^n$ we will let ${\rm Gr}^n_\lambda(Z)$ denote the subset of ${\rm Gr}^n_\lambda$ for which the corresponding 
linear subspace has a non-empty intersection with $Z$.
A $\langle \lambda\rangle$-point of $Z$ is in the intersection of $Z$
with a $\langle 0\rangle$-point (i.e. rational point) of 
${\rm Gr}^n_\lambda(Z)$. 
We will apply a version of the Counting Theorem to rational points of 
${\rm Gr}^n_\lambda(Z)$, which is some definable subset of 
${\rm Gr}^n_\lambda$.
What we need to show is that algebraic parts 
of ${\rm Gr}^n_\lambda(Z)$ ``come from'' anomalous parts of $Z$.

If $L\in {\rm Gr}^n_\lambda$ then $S(L)$ denotes the corresponding
linear subvariety in $\RR^n$. More generally, if 
$M\subset{\rm Gr}^n_\lambda$ then
$S(M)$ denotes the Zariski closure of the union of the $S(L)$, for $L\in M$.

\begin{proof}
We will consider rational points in ${\rm Gr}^n_\lambda(Z)$. 
Though ${\rm Gr}^n_\lambda$ is projective, we may consider it to be a union of finitely many affine pieces,
and so we may assume ${\rm Gr}^n_\lambda\subset\RR^N$ 
for some $N=N(n,\lambda)$. Given $\epsilon$, 
we apply the Counting Theorem to conclude that
${\rm Gr}^n_\lambda(Z)(\QQ,T)$ is contained in 
$c(Z,\lambda,\epsilon)T^\epsilon$ blocks from some finite collection 
of block families. The degrees of the block families are all bounded 
by some $b(k, N,\epsilon)\ge 1$ independent of the 
subset of $\RR^N$ except for dimension. 
This is our $b(k, n,\lambda,\epsilon)$.
Accordingly, we can replace $Z$ by $Z-Z^{[\lambda,b]}$, and after 
this replacement, we have that $Z^{[\lambda,b]}=\emptyset$.

Consider a point $L\in{\rm Gr}^n_\lambda(Z)$, 
so $L\cap Z\ne\emptyset$. 
If $L\cap Z$ has positive dimension $\kappa$, then this 
immediately gives us a $(\kappa,\lambda,1)$-block of 
deficiency $\lambda-\kappa\le\lambda$
(even $<$), and so the positive dimensional parts are 
contained in $Z^{[\lambda,1]}$.
So we may assume: such intersections are all finite and hence 
(by uniformity properties in o-minimal structures) of cardinality
uniformly bounded by $C$ (depending on $Z,\lambda$) for
all $L$.

Given $\epsilon$, there is a $\beta$ such that 
${\rm Gr}^n_\lambda(Z)-{\rm Gr}^n_\lambda(Z)^{[0,\beta]}$ 
has $\ll_\epsilon T^\epsilon$ rational  points.
Moreover, ${\rm Gr}^n_\lambda(Z)(\QQ,T)$ is contained in 
$\ll_\epsilon T^\epsilon$ blocks of degree $\le b$.

Consider some $(0,w,b)$-block $U$ of ${\rm Gr}^n_\lambda(Z)$ 
contained in some closed algebraic $A$ of degree $\le b$. 
Each $S(L)$ for $L\in U$ intersects $Z$ in $\le C$ points. 
There may be some points $P\in Z$ that indeed lie on every 
$S(L)$, with $L\in U$. There can be at most $C$ of these.

%\medskip
%\noindent
%{\bf Claim:\/}
\begin{claim}
Let $P_i$ be the points of $Z$ that lie on every $S(L)$, with $L\in U$. Then,
the set of $L\in U$ that intersect $Z$ in another point not 
among the $P_i$ has lower dimension than $U$.
\end{claim}
%\medskip

\begin{proof}[Proof of the claim] If not, there is a definable function on some subset
$U'\subset U$ of full dimension taking an element $L$ to an element of 
$S(L)\cap Z$ whose value is never a $P_i$. On some full dimensional subset of regular points $U'$, we will have that $f$ is differentiable. 
Let $Q$ be such a point. Suppose the derivative of the function is non-zero in some direction. Intersecting $U'$ with a suitable linear
subvariety, we get an algebraic curve segment $C$ in $U'$ of degree $\le b$ (through the point in a direction of non-zero derivative) such that the corresponding point  on $Z$ is non-constant. Then $S(C)$ has dimension $\le \lambda+1$, as it is a curve of $\lambda$-folds, has degree $\le b$,
and intersects $Z$ in a set of positive dimension.
This gives rise to a $(1, \lambda+1,b)$-block in $Z$, i.e. a $[\lambda,b]$-block,
which is impossible. So the derivative is zero at all such points, 
in every direction. Then the function is constant, giving a new point lying
on all the $S(L)$, with $L\in U$, which is a contradiction.
This proves the claim.
\end{proof}
%\medskip

So given such a block, apart from some ``fixed'' intersections, 
any ``extra'' (non-constant) intersections occur only on a lower 
dimensional set, definable and indeed
in a definable family for $U$ coming from a definable family.
Let $W$ be such an exceptional set. We consider $W(\QQ,T)$, 
which is contained in at most $\ll_\epsilon T^\epsilon$ 
blocks of degree $b$. 
Those blocks have again some fixed intersections etc.

The dimensions of the exceptional sets $W$ decrease at 
each stage, and eventually
we have  all the $\langle \lambda\rangle$-points of
$Z$ coming from ``fixed'' intersections of at most
$\ll_\epsilon T^\epsilon$ blocks, each element $L$ of which has some 
bounded number of intersections with $Z$.
This proves the theorem.

\end{proof}

\section*{Acknowledgements}

The authors are grateful to the referee of the paper for the useful comments and remarks.
LC, DM and UZ thank the ERC-Grant No. 267273 who supported their research. JP acknowledges with thanks that his research was supported in part by a grant from the EPSRC entitled “O-minimality and diophantine geometry”, reference EP/J019232/1.

\small 
\bibliographystyle{amsalpha} 
\bibliography{bibliography} 

\end{document}